\documentclass{article}
\usepackage{graphics}
\usepackage{ifthen}
\usepackage{amsmath,amsfonts,amstext,amscd,bezier,amssymb,a4,enumerate,epsfig,theorem,psfrag,subfigure}
\usepackage[latin1]{inputenc}
\usepackage{graphicx}
\newcommand{\Ran}{{\operatorname{Ran}}}

\newcommand{\RR}{{\mathbb{R}}}

\newcommand{\tz}{{\tilde{z}}}
\newcommand{\tv}{{\tilde{v}}}

\newcommand{\tF}{{\tilde{F}}}

\newcommand{\Vp} {\ensuremath{V_{\scriptscriptstyle X}}}%
\newcommand{\Vd} {\ensuremath{V_{\scriptscriptstyle Y}}}%
\newcommand{\Wp} {\ensuremath{W_{\scriptscriptstyle X}}}%
\newcommand{\Wd} {\ensuremath{W_{\scriptscriptstyle Y}}}%

\newcommand{\optpar} [2] []% first argument optional within parentheses
  {\ensuremath{#2\ifempty{#1} {} {\!\left(#1\right)}}}
\newcommand{\optsub} [2] []% first argument optional becomes subindex
  {\ensuremath{#2\ifempty{#1} {} {_{#1}}}}
\newcommand{\optsuper} [2] []% first argument optional becomes superindex
  {\ensuremath{#2\ifempty{#1} {} {^{#1}}}}

\newlength{\boxwd}
\newlength{\boxht}

\newcommand{\ifempty} [1] {\ifthenelse{\equal {#1} {} }}

\newcommand{\dOmega}{\ensuremath{\partial\Omega}}
\newcommand{\Lp} [2] [] {\optpar[#1] {L^{#2}}}                   % optional first parameter is domain
\newcommand{\Lone} [1] [] {\Lp[#1]{1}}
\newcommand{\Ltwo} [1] [] {\Lp[#1]{2}}
\newcommand{\Linfty} [1] [] {\Lp[#1]{\infty}}
\newcommand{\Hk} [2] [] {\optpar[#1] {H^{#2}}}
\newcommand{\Hone} [1] [] {\Hk[#1]{1}}
\newcommand{\Hnegone} [1] [] {\Hk[#1]{-1}}
\newcommand{\Htwo} [1] [] {\Hk[#1]{2}}
\newcommand{\Hzerok} [2] [] {\optpar[#1] {H_0^{#2}}}
\newcommand{\Hzeroone} [1] [] {\Hzerok[#1]{1}}

\newcommand{\R} {\ensuremath{\mathbb{R}}}
\newcommand{\tf} {\tilde{f}}

\newcommand{\lap} {\Delta} % laplatian
\newcommand{\seminorm} [1] [\,] {\ensuremath{|{#1}|}}    % generic seminorm
\newcommand{\semnrm} [2] [\,]% 1st argument (optional) is argument, second name/domain
  {\ensuremath{\seminorm[{#1}]_{\scriptscriptstyle#2}}}
\newcommand{\nrm} [2] [\,]%
  {\ensuremath{|\!|{#1}|\!|_{{\scriptscriptstyle#2}}}}

\newcommand{\Lonenrm} [2] []%
  {\nrm[{#2}] {\Lone[#1]}}
\newcommand{\Ltwonrm} [2] []%
  {\nrm[{#2}] {\Ltwo[#1]}}
\newcommand{\Linftynrm} [2] []%
  {\nrm[{#2}] {\Linfty[#1]}}
\newcommand{\Honenrm} [2] []%
  {\nrm[{#2}] {\Hone[#1]}}
\newcommand{\Htwonrm} [2] []%
  {\nrm[{#2}] {\Htwo[#1]}}

\newcommand{\negnrm} [1] []%
  {\nrm[{#1}] {-1}}
\newcommand{\zeronrm} [1] []%
  {\nrm[{#1}] {0}}
\newcommand{\onenrm} [1] []%
  {\nrm[{#1}] {1}}
\newcommand{\twonrm} [1] []%
  {\nrm[{#1}] {2}}
\newcommand{\inftynrm} [1] []%
  {\nrm[{#1}] {\infty}}

\newcommand{\onesemnrm} [1] []%
  {\semnrm[{#1}] {1}}

\newcommand{\Honesemnrm} [2] []%
  {\semnrm[{#2}] {\Hone[#1]}}
\newcommand{\Htwosemnrm} [2] []%
  {\semnrm[{#2}] {\Htwo[#1]}}

\newtheorem{lem} {Lemma}
\newtheorem{prop} {Proposition}

\newtheorem{teo} {Theorem}

\def\qed{\unskip\nobreak\hfil\penalty50\hskip1.75em\null\nobreak\hfil
$\blacksquare$ {\parfillskip=0pt \finalhyphendemerits=0 \par}\medbreak}

\newcommand\capsize{\relax}

\title{Geometric aspects of  Ambrosetti-Prodi operators \\   with Lipschitz nonlinearities}
\author{Carlos Tomei and André Zaccur\vspace{0.5cm} \\
Departamento de Matemática, PUC-Rio\footnote{ R. Mq.  S. Vicente 225,
Rio de Janeiro 22451-900,  Brazil. email: tomei@mat.puc-rio.br.} \vspace{0.5cm} \\
}
\date{}
\begin{document}
\maketitle

\centerline{\it Dedicated to Bernhard, with affection and admiration}

\begin{abstract}
Let the function $u$ satisfy  Dirichlet boundary conditions on a bounded domain $\Omega$.
What happens to the critical set of the Ambrosetti-Prodi operator $F(u) = - \lap u - f(u)$ if the nonlinearity is only a Lipschitz map? It turns out that many properties which hold in the smooth case are preserved, despite of the fact that $F$ is not even differentiable at some points. In particular, a global Lyapunov-Schmidt decomposition of great convenience for numerical solution of $F(u) = g$ is still available.
\end{abstract}

\medbreak

{\noindent\bf Keywords:}  Semilinear elliptic equations, Ambrosetti-Prodi theorem.

\smallbreak

{\noindent\bf MSC-class:} 35B32, 35J91, 65N30.

\section{Introduction}\label{sec:intro}

A familiar set of hypotheses for the celebrated Ambrosetti-Prodi theorem is the following.
Let $\Omega \subset \RR^n$ be an open, bounded, connected domain with smooth boundary $\dOmega$ and denote by $ 0 < \lambda_1 < \lambda_2 \le \ldots $
 the eigenvalues of the free Dirichlet Laplacian $- \Delta$
on $\Omega$. Let $f: \RR \to \RR$ be a smooth, strictly convex function, with asymptotically linear derivative so that
\[\Ran \ f' = (a,b) \, ,  \quad a < \lambda_1 < b <\lambda_2 \,.\]

Under such hypotheses, the theorem states that the equation
\begin{equation}\label{eq:problem}
 F(u) =  - \lap u - f(u) = g, \quad u|_{\partial \Omega} = 0
\end{equation}
for, say, $g \in C^{0,\alpha}(\Omega)$,  has (exactly) zero, one or two solutions in $C^{2,\alpha}(\Omega)$.

\subsection{A first approach --- locating $C$ and $F(C)$}

In the original arguments (\cite{AP}, \cite{MM}), a fundamental role is played by the critical set $C$ of  $F: X = C^{2,\alpha}_D(\Omega) \to Y = C^{0,\alpha}(\Omega)$. Here $C^{2,\alpha}_D(\Omega)$ is the subspace of functions of $C^{2,\alpha}(\Omega)$ satisfying Dirichlet boundary conditions. The proof follows a few steps, which follow from the inverse function theory and the characterization of fold points.

\begin{enumerate}
\item{ The critical set $C \subset X$ is  a hypersurface; every critical point is a fold.}
    \item{ $F$ is proper, its restriction to $C$ is injective and $F^{-1}(F(C)) = C$.}
        \item{The spaces $X - C$ and $Y - F(C)$ have two components. Each component of $X - C$ is taken injectively to the same component of $Y- F(C)$.}
\end{enumerate}

\subsection{A global Lyapunov-Schmidt decomposition}

Berger and Podolak  came up with a different approach \cite{BP}, which is easier to phrase for $F: X \to Y$ between Sobolev spaces,  $X = \Hzeroone[\Omega]$ and $Y = \Hnegone[\Omega] \simeq \Hzeroone[\Omega]$. Their main result is the construction of a global Lyapunov-Schmidt decomposition for $F$. Let $\varphi_1$ be the (positively normalized) eigenfunction associated to the ground state $\lambda_1$,  set $\Vp = \Vd = \ \langle \varphi_1 \rangle$
and consider the orthogonal decompositions
\[ X = \Vp \oplus \Wp \ \hbox{  and  } \ Y = \Vd \oplus \Wd\]
into {\it vertical} and {\it horizontal} subspaces. With a different vocabulary, their proof essentially goes through the verification of the following properties, by making use of spectral estimates on the Jacobians $DF(u): X \to Y$.
\begin{enumerate}
\item{Horizontal affine subspaces of $X$ are taken by $F$ to {\it sheets}.}
\item{ The inverse under $F$ of vertical affine subspaces of $Y$ are {\it fibers}.}
\item {Sheets are {\it essentially flat}, fibers are {\it essentially steep}.}
\end{enumerate}

An affine horizontal subspace of $X$ is a set of the form $x + \Wp$, for a fixed $x \in X$; an affine vertical subspace of $Y$ is of the form $y + \Vd$, for $y \in Y$.
Sheets are graphs of smooth functions from $\Wd$ to $\Vd$ and fibers are graphs of smooth functions from $\Vp$ to $\Wp$. The third property states that the inclination of the tangent spaces to sheets,
with respect to horizontal subspaces, and to fibers, with respect to vertical subspaces, is uniformly bounded from above.

In the words of \cite{CT}, $F$ is a {\it flat} map. In the Ambrosetti-Prodi case, vertical spaces and fibers are one dimensional. More generally, the dimension equals the number $k$ of eigenvalues of $- \Delta_D$ in the set $\overline{f'(\RR)}$ (we suppose non-resonance, i.e., the asymptotic values of  $\Ran f'$ are not eigenvalues) and similar results  still hold.

\subsection{Sheets and fibers allow for weaker hypothesis}

There is an interesting bonus obtained from considering this global Lyapunov-Schmidt decomposition. For many nonlinearities $f$, the related nonlinear map
\[F:X \to Y, \quad F(u) = - \Delta u - f(u)\]
is not proper, and part of the technology related to degree theory simply breaks down. The Ambrosetti-Prodi hypothesis yield properness of $F$, but it is not really essential to most of what we want to do. The first section of the paper is dedicated to explicit examples of Lipschitz  nonlinearities for which properness does not occur --- there are functions $g \in Y$ for which $F^{-1}(g)$ contains a full half-line of functions in $X$. From such examples, we may obtain smooth nonlinearities with the same property, but we give no details.

The results in the paper indicate that in most directions, properness holds. What we mean by this is something similar to the fact that even when a  map $F$ does not have an invertible Jacobian $DF(u)$ at a point $u$, it  may still have a subspace $L(u)$ on which the restriction of $DF(u)$ acts injectively --- this is how bifurcation equations come up: they concentrate in a possibly small subspace $S(u)$ transversal to $L(u)$ the difficulties which are not resolved by the linearization at $u$. What we shall see is that the domain $X$ of $F$ splits into special (nonlinear) surfaces of finite dimension $k$, on which properness may break down, which have for tangent spaces the subspaces $S(u)$ when $u$ is critical, but on transversal directions to these surfaces, the horizontal affine subspaces, properness is always present. As we shall see, counting or computing solutions of the differential equation $F(u) = g$ boils down to a finite dimensional problem, simplifying both (abstract) analysis and numerics.

\subsection{A related numerical algorithm to  solve the PDE $F(u)=g$}

Smiley and Chun (\cite{S},\cite{SC}) showed that an analogous global Lyapunov-Schmidt decomposition exists for appropriate non-autonomous nonlinearities $f(x,u(x))$ and emphasized its relevance for numerical analysis: one might solve $F(u)=g$ by restricting $F$ to the fiber $\alpha_g$ containing $F^{-1}(g)$.
%(\cite{smiley:1996}, \cite{smiley:1998},\cite{smiley:2000}).
In \cite{CT}, this project is accomplished: given a right hand side $g$, one first obtains numerically a point in $\alpha_g$, which in the Ambrosetti-Prodi case is a curve, and then proceeds to search for solutions by moving along it. The algorithm is sufficiently robust to handle more flexible nonlinearities: it does not require convexity  of $f$ or properness of $F$, and the range of $f'$ may include other (finite) sets of eigenvalues of $- \Delta_D$.

\subsection{Lipschitz nonlinearities}

Now, what happens when $f$ is not smooth anymore, but, say, Lipschitz? In particular, this is the scenario considered in the proof of the so called one dimensional Lazer-McKenna conjecture (\cite{LM1}, \cite{LM2}) by Costa, Figueiredo and Srikanth in \cite{CFS}.
We state the result, for the reader's convenience. Let $X=H^2_D([0,\pi])$ be the Sobolev space of functions satisfying
Dirichlet boundary conditions  with square integrable second derivatives. Recall that  $u \mapsto -u''$ acting on $X$ to $Y = L^2([0,\pi])$ has eigenvalues  $\lambda_k = k^2$, $k=1,2,\ldots$ with corresponding eigenfunctions $\sin (kx)$. Take $f: \RR \to \RR$, a strictly convex smooth function $f$  with asymptotic values $a$ and $b$ for its derivative $f'$ satisfying
\[ a < 1 , \quad \lambda_k = k^2 < b < (k+1)^2 = \lambda_{k+1}.\]
Then the equation $F(u) = -u'' - f(u) = -t \sin x$, $u(0) = u(\pi)=0$, has exactly $2k$ solutions for $t>0$ sufficiently large.

The argument in \cite{CFS} considers the nonlinearity $\tf$ given by $\tilde{f'}(x) = a \hbox{ or } b$, depending if $x <0$ or $x>0$. The related operator $\tF(u) = -u'' - \tf(u)$ now requires some care: it stops being differentiable everywhere and the usual differentiable  normal forms at regular points and folds break down.

In this paper, we show that when $f$ is merely Lipschitz, for appropriate conditions on the boundary of $\Omega$, the operator $F$ is still flat, in the sense that the global Lyapunov-Schmidt decomposition still holds, and sheets and fibers are still available as graphs of Lipschitz functions. A word of caution: piecewise linear nonlinearities may yield continua of points on which the map $F$ takes a unique value. Such sets necessarily lie in a single fiber. More, the numerical analysis of the PDE $F(u) = g$ presented in \cite{CT} is still valid, after minor modifications.

We take this material to be an intermediate step towards a more geometric description of the operators of Hamilton-Jacobi-Bellman type, as studied by Felmer, Quaas and Sirakov in \cite{FQS}.

The authors gratefully acknowledge support from CAPES, CNPq and FAPERJ.

\section{Some cautionary examples}

For a box $\Omega \subset \RR^n$, we consider again the differential equation
\begin{equation}
 F(u) =  - \lap u - f(u) = g, \quad u|_{\partial \Omega} = 0
\end{equation}
with a Lipschitz nonlinearity $f(x)$. Once $f$ is  piecewise linear, there may be whole (straight line) segments on the domain restricted to which $F$ is actually constant. This happens already for the one-dimensional case. Let $\varphi_1$ be the (positive) ground state associated to eigenvalue $\lambda_1$ and take $a < \lambda_1 < b $. Suppose that $\Ran \ \varphi_1 = [0, M]$ and define $f(x)$ to be continuous, with derivatives equal to $a$, $\lambda_1$ and $b$ in the intervals $[-\infty, 0],
[0, M]$ and $[M, \infty]$: clearly, $F(t \varphi_1) = 0$, for $t \in [0,1]$.

\subsection{Nonlinearities $f$ with derivatives taking two values, $n=1$}
In a similar vein, we now provide examples of segments on which $F$ is constant for the nonlinearity $f(x) = ax $ or $bx $, for $x <0$ or $x>0$, with the property that, for special values of $a$ and $b$,  there are right hand sides
$g$ ($g \equiv 0$ is an example) with the property that $F^{-1}(g)$ contains a (straight) half-line of solutions. In particular, the map $F: X = H^2(\Omega) \cap H^1_0(\Omega) \to L^2(\Omega)$ is not proper and the equation has solutions which are not isolated.

\begin{figure}[ht]
\begin{center}
\epsfig{height=30mm,file=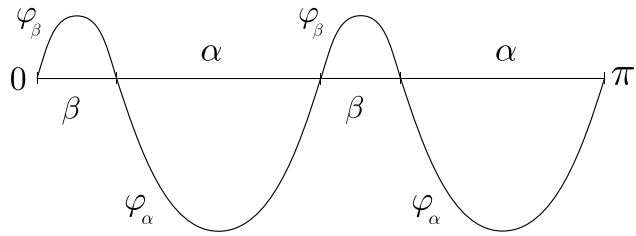}
\end{center}
\caption{\capsize A half-line of solutions}
\label{fig:wil}
\end{figure}

We begin with the one dimensional case, $n=1$.  Set $\Omega = I = [0,\pi]$. Split $I$ into $k$ closed intervals $I_i, i=1,..,k$ joined at their ends, of two different lengths, $I_{odd \ i} = \beta$, $I_{even \ i} = \alpha$. In the figure, $k=4$. The smallest eigenvalues for the operator $u \mapsto -u''$ with Dirichlet conditions in an interval of sizes $\beta$ and $\alpha$ are respectively
\[\lambda_\beta = \big( \frac{\pi}{\beta} \big)^2 \quad \hbox{ and } \quad
\lambda_\alpha = \big( \frac{\pi}{\alpha} \big)^2 \ ,\]
with positive (normalized) eigenfunctions $\varphi_\beta$ and $\varphi_\alpha$.
We set $b = \lambda_\beta$ and $a= \lambda_\alpha$ and construct a solution $\psi$ by juxtaposing multiples of $\varphi_\beta$ and $\varphi_\alpha$ as shown in the figure. On $I_1$, one may take $p \varphi_\beta$, for arbitrary $p >0$. On $I_2$, the (negative) multiple of $\varphi_\alpha$ is determined by matching the first derivative --- recall that $\psi \in X = H^2(I) \cup H^1_0(I)$, so $\psi'$ is absolutely continuous. The procedure extends to the remaining intervals in a unique fashion. We have to make sure that the total length of the intervals $I_i$ equals $\pi$.
Thus, for example, in the simplest case $k=2$, we must have
\[ \beta + \alpha = \pi \quad \Longleftrightarrow \frac{1}{\sqrt{b}} + \frac{1}{\sqrt{a}} = 1.\]
For any value $a \in (1, 4)$, there is a (unique) $b$, which turns out to be in $(4, \infty)$ which solves this equation. Said differently: any interval $[a,b]$ containing $\lambda_2 = 4$ for which $a,b$ are not eigenvalues of the free problem admits a half-line of solutions of the equation above. For different numbers of intervals, one shows half-lines of solutions for any $a \in (\lambda_k, \lambda_{k+1})$ and appropriate $b \in (\lambda_{k+1}, \infty)$.

Alas, the only situation for which this argument does not provide a half-line of solutions $F^{-1}(0)$ is $a < \lambda_1$, the Ambrosetti-Prodi case. There are strong evidences that in this case there are no continua of solutions $F^{-1}(g)$, but we have no real proof.

Clearly, one may replace $g\equiv 0$ by any $g$ defined piecewise on intervals $I_i$ as functions in the range of $u \mapsto -u'' - f(u)$ (Dirichlet conditions on $I_i$) acting on positive functions  restricted to $I_i$. Thus, for $k$ intervals, the set of such $g$ is a vector subspace of $L^2(I)$ of codimension $k$. This construction ascertains that $g$ is in the range of $F$ (now considered in the full interval $I$), so that $g = F(u_0)$. By linearity on each interval $I_i$, adding an homogeneous solution $\psi \in F^{-1}(0)$ gives rise to $u_0 + \psi \in F^{-1}(g)$.

These ideas also suffice to prove that there is no nontrivial function in $[0,\pi]$ which is taken to $0$ if $a < \lambda_1 < b$.

\subsection{The case of arbitrary $n$}

We now consider the case $n=2$, the general situation being similar. We now have $\Omega = I \times I$ (rectangles would work also): just separate variables and proceed. Let $\psi(y)$ as before, solving $-\psi_{yy} - f(\psi)=0$, where $f$ is constructed from appropriate $a$ and $b$, and let $\varphi(x) \geq 0$ be the ground state for Dirichlet conditions on $I$, so that $-\varphi_{xx} = \varphi(x)$. The product
$\tilde{\psi}(x,y)= \varphi(x) \psi(y)$ satisfies
\[ - \tilde{\psi}_{xx} - \tilde{\psi}_{yy} - f(\tilde{\psi}) =
\tilde{\psi} + \varphi(x) (- \psi_{yy} - f(\psi)) = \tilde{\psi}\]
so that $\tilde{\psi}$ and its positive multiples solve
\[ -u_{xx} - u_{yy} - \tilde{f}(u) = 0, \quad u|_{\partial \Omega} = 0,\]
for $\tilde{f}(x) = f(x) +x$.

\section{Geometry of Lipschitz maps}

%Take $\Omega \subset \RR^n$ to be a bounded set with $C^2$ boundary (this is too much --- the results hold for products of intervals: we provide more details after setting up notation).
Set $Y=L^{2}(\Omega)$ with inner product $\langle u, v \rangle_0 = \int_\Omega uv$ and norm $\| u \|_0$. Also let $X=H^2(\Omega)\cap H^1_0(\Omega)$, with inner product $\langle u,v \rangle = \langle -\Delta u \ ,\ -\Delta v\rangle_0$
and norm $\|u\|_2$.

We always consider sets $\Omega$ for which $- \Delta: X \subset Y \to Y$ is a self-adjoint isomorphism --- we call such domains $\Omega$ {\it appropriate}. Also the same operator should have $C_0^\infty(\Omega)$ as a core, i.e. it is essentially self-adjoint in this domain.
From the spectral theorem, there is an orthonormal basis of (Dirichlet) eigenfunctions $\varphi_i \in X, \ \|\varphi_i\|_0 = 1$, satisfying $-\Delta \varphi_i = \lambda_i \varphi_i$. Eigenfunctions associated to different eigenvalues are orthogonal with respect to both inner products.
Concretely, one might take $\Omega$ to be a convex set or require $\partial \Omega$ to be $C^{1,1}$ (\cite{S},\cite{G}).
 Notice that, from standard results in spectral theory, operators
 \[ T : X \subset Y \to Y, \quad T u = -\Delta u - q u,\] for bounded real potentials $q$, are still self-adjoint with an orthonormal basis of eigenfunctions.

We assume that the nonlinearity $f: \RR \to \RR$ is Lipschitz,  and $f'$ takes values in an interval $[a,b]$ with the property that the bounds $a$ and $b$ are not eigenvalues $\lambda_i$. For this part of the paper, we make no assumptions about convexity for $f$. Notice that $a$ and $b$  do not have to be the asymptotic values of $f'$, a degree of freedom which is convenient for numerical analysis.

For starters, $F: X \to Y$ given by $F(u) = - \Delta u - f(u)$ is a well defined map --- it suffices to check that $f(u)  \in Y$. This follows from the easy lemma below.

\begin{lem}
Say $f: \RR \to \RR$ is Lipschitz. Then the map
$\hat{f}: Y=L^2(\Omega) \to Y$ given by $\hat{f}(u) = f \circ u $ is also well defined and Lipschitz with the same constant.
\end{lem}
\begin{proof} Take first $u, v$ continuous functions.
Since $f$ is $M$-Lipschitz, it is absolutely continuous, so that
\[
|f(u(x))|
 =|f(0)+u(x)\int^1_0{f'(tu(x))dt} \ |
 \leq |f(0)|+M|u(x)|,   x\in \Omega,
\]
and, since $\Omega$ is bounded, we have $f(u)\in Y$. Similarly,
applying the fundamental theorem of calculus to the function
$\varphi(t) = f(tu(x) + (1-t)v(x))$, one obtains
\[ \| f(u(x)) - f(v(x)) \|_0 \le \int_0^1 |f'(tu(x) + (1-t) v(x))| dt \ \|u(x) - v(x)\|_0.\]
Now take Cauchy sequences of continuous functions converging to arbitrary functions in $Y$: the estimates above extend to the required $L^2$ estimates.
\end{proof}
\qed

\subsection{The main result: $F_v:W_X \to W_X$ is an isomorphism}

We now describe an orthogonal decomposition of $X$ and $Y$. Take $\Omega \subset \R^n$ to be a bounded appropriate domain and let $\Lambda_f = \{ \lambda_i\}_{i \in I}$ be the set of eigenvalues $\lambda_i$ in $(a,b)$. The
{\it vertical subspaces}
$ V_X =V_Y$  equal the invariant subspace associated to $\Lambda_f $ and $V_X \subset X, V_Y \subset Y$.
The {\it horizontal} subspaces are
$W_{X} =V_X^{\perp}\subset X $ and $
W_{Y} =V_Y^{\perp}\subset Y$
where orthogonality takes into account the (different) inner products in $X$ and $Y$. These induce orthogonal decompositions
$X = W_X \oplus V_X, \quad Y = W_Y \oplus V_Y $ and corresponding orthogonal projections $P_Y$ and $Q_Y$ from $Y$ to $W_Y$ and $V_Y$. Finally, we consider {\it affine horizontal subspaces} in $X$, which are sets of the form $x + W_X$, for a fixed $x$, and {\it affine vertical subspaces} in $Y$, of the form $y + V_Y$.

We need a label for this construction:  a nonlinearity $f$ induces an {\it $I$-decomposition} $X = W_X \oplus V_X, \ Y = W_Y \oplus V_Y $ associated to bounds $a$ and $b$.

\begin{teo}\label{Fvv} Let $\Omega$ be an appropriate domain, $f: \RR \to \RR$ Lipschitz,  $\Ran f' \subset [a,b]$, where $a$ and $b$ are not eigenvalues $\lambda_i$.
 and  $X = W_X \oplus V_X, \ Y = W_Y \oplus V_Y $ be the $I$-decomposition specified above. For $v \in V_X$, let $F_v: W_X \to W_Y$ be the horizontal projection of the restriction  of $F$ to the affine subspace $v+ W_X$,
 $F_v(w) = P_Y F(w+v)$. Then $F_v$ is a bi-Lipschitz homeomorphism. The Lipschitz constants for $F_v$ and $F_v^{-1}$  are independent of $v$.
\end{teo}

\begin{proof}
For $\gamma=(a+b)/2 $, set $ \tilde{f}(x)=f(x) - \gamma x$. Then   $T : W_X \to W_Y$ given by $u\to -\Delta u - \gamma u$ is  well defined and invertible, with  eigenvalues $\lambda_j - \gamma$ with $j \notin I$. Let $\lambda_m - \gamma$ be the eigenvalue of $T$ of smallest absolute value:
clearly,
\[|| T^{-1}|| = |\lambda_m - \gamma|^{-1} > | a - \gamma | =  b - \gamma. \]
For $u = w + v, w \in W_X, v \in V_X$, we have
\[
F_v(w)  =P_Y[-\Delta(w+v)-f(w+v)]
 = T w -P_Y \tilde{f}(w+v).
\]
The composition
$F_v\circ T^{-1}: W_Y \to W_Y$ is of the form $I-K_v$, where $K_v(w)=P_Y\tilde{f}(T^{-1}w+v)$. We show that $K_v: W_Y \to W_Y$ is a contraction with constant uniformly bounded away from 1.  For $w, \tilde{w} \in W_Y \subset L^2(\Omega)$,
\[ \|K_v(w) - K_v(\tilde{w})\|_0   \leq \  \|\tilde{f}(T^{-1}w+v)-\tilde{f}(T^{-1}\tilde{w}+v)\|_0 \]
\[\leq  (b-\gamma) \|T^{-1}(w-\tilde{w})\|_0
 \leq \  \frac{b-\gamma}{|\lambda_m-\gamma|} \ \|w-\tilde{w}\|_0 = c \ \|w-\tilde{w}\|_0. \]
Since $b$ is not an eigenvalue $\lambda_j$, the Lipschitz constant $c$ is uniformly bounded away from 1. From the Banach contraction theorem, $F_v: W_Y \to W_Y$ is a homeomorphism and standard estimates using the familiar formula
\[ (I - K_v)^{-1} = I + K_v + K_v \circ K_v + K_v \circ K_v \circ K_v + \ldots  \]
show that $(F_v)^{-1}$ is Lipschitz, where the constant depends on $c$ and not on $v$.
\end{proof}
\qed

In particular, if $[a,b]$ does not contain eigenvalues $\lambda_i$, the result above recovers the Dolph-Hammerstein theorem  for Lipschitz nonlinearities (\cite{D}, \cite{H}).

\bigskip

\subsection{Sheets and fibers make sense for Lipschitz nonlinearities}

We are ready to extend to the Lipschitz context  the global Lyapunov-Schmidt decomposition which is known for the smooth case, from the works of Berger and Podolak and Smiley. Consider the following diagram.
\[
\begin{array}{ccl}
{X = W_X \oplus V_X}&
\stackrel{{\scriptstyle F}}{\longrightarrow}&
{Y = W_Y \oplus V_Y} \\
   {\scriptstyle \Phi^{-1}=( F_v, Id)}\searrow &
      &
\nearrow{\scriptstyle \tilde{F}=F\circ\Phi=(Id, \phi)}\\
    &{Y = W_Y \oplus V_Y}

&  \\
  \end{array}
\]

Thus the change of variables $\Phi$ yields $\tilde{F}(w,v) = F \circ \Phi (w,v) = (w, \phi(w,v))$, from which we will derive some convenient geometric properties. We first clarify a technicality: $\Phi$ is indeed a global change of variables in the Lipschitz category. In $W_X \oplus V_X $, we use the norm obtained by adding the norms in each coordinate.

\begin{figure}[ht]
\begin{center}
\epsfig{height=50mm,file=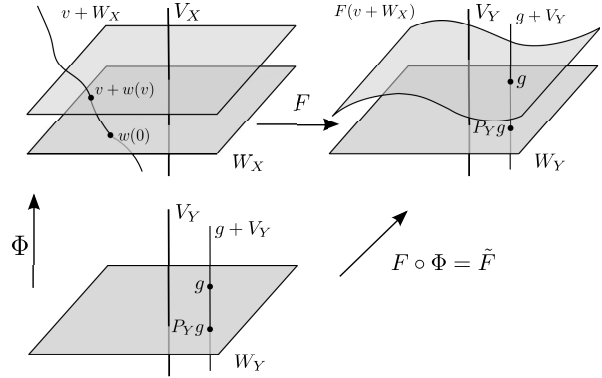}
\end{center}
\caption{\capsize The change of variables}
\label{fig:wil}
\end{figure}

%\begin{figure}
%\begin{center}
%\includegraphics[scale=0.8]{figura2.pdf}
%\caption{A composi\c{c}\~ao $F\circ\Phi$ preserva a componente horizontal. Al\'em disso, o conjunto $\{w+h(t)\varphi_1\}$ \'e uma semi-reta, conforme demonstramos na pr\'oxima se\c{c}\~ao.}
%\end{center}
%\end{figure}

\begin{prop} The map $\Phi = ((F_v)^{-1}, Id):  Y=W_Y \oplus V_Y \to X=W_X \oplus V_X$ is a bi-Lipschitz homeomorphism.
\end{prop}

\begin{proof} The invertibility of $\Phi$ follows from the previous theorem. We use some elementary facts. The identity map $Id:(V_X, \|.\|_2) \to (V_Y, \|.\|_0)$ between normed spaces of the same finite dimension is bi-Lipschitz. Also,
since $W_X$ and $V_X$ are orthogonal (in $L^2$ and $H^2$), $\|w\|_2 + \| v\|_2 \leq 2 \| w \pm v \|_2$
for $w \in W_X$ and $v \in V_X$. To show that $\Phi^{-1}$ is Lipschitz, take $w+v, \tilde{w} + \tilde{v} \in X$. For appropriate constants $C$, $\tilde{C}$,

\[ \|\Phi^{-1} (w+v)-\Phi^{-1} (\tilde{w}+\tilde{v})\|_0 = \| F_v(w)-F_{\tilde{v}}(\tilde{w})\|_0 +  \|v-\tilde{v}\|_0\]
\[ \leq  \|-\Delta (w- \tilde{w}) - P_Y(f(w+v)-f(\tilde{w}+\tilde{v}))\|_0 + C \|v-\tilde{v}\|_2\]
\[
\leq  \|w-\tilde{w}\|_2 + \| f(w+v)-f(\tilde{w}+\tilde{v})\|_0 + C \|v-\tilde{v}\|_2  \leq  \tilde{C} \|w+v-(\tilde{w}+\tilde{v})\|_2, \]
where the last inequality follows from Lemma 1.

We obtain a Lipschitz estimate for  $\Phi = ( (F_v)^{-1}, Id): W_Y \oplus V_Y \to W_X \oplus V_X$. Take $z+v, \tilde{z}+ \tilde{v} \in Y = W_Y \oplus V_Y$. Then
\[ \| \Phi(z+ v) - \Phi(\tz + \tv) \|_2 \leq
\| F_z^{-1}(v)  - F_{\tz}^{-1}(v)\|_2 + \|F_{\tz}^{-1}(v)- F_{\tz}^{-1}(\tv)\|_2 + \| v - \tv\|_2\]

Again from finite dimensionality of $V_X= V_Y$, there is an estimate of the form $\| v - \tv\|_2 \leq C \| v - \tv\|_0$. We also have $\|F_{\tz}^{-1}(v)- F_{\tz}^{-1}(\tv)\|_2 \leq C \| v - \tv\|_0$ from the proof of the previous theorem. The first term is handled in a similar fashion.
\end{proof}
\qed

The picture should help putting pieces together. Here, $\dim V_X= \dim V_Y = 1$, as in the  Ambrosetti-Prodi theorem: the convex span of $\Ran f'$ ($f$ is Lipschitz!) contains only the eigenvalue $\lambda_1$. The map $F$ takes an affine horizontal subspace $v + W_X$ to a sheet, and the inverse of the vertical affine subspace $g + V_Y$ is a fiber, which crosses $W_X$ at $w(0)$ and $v + W_X$ at $v + w(v)$, in the notation of the proof above. Clearly, sheet and fiber are graphs, as stated above.
The change of variables $\Phi$ preserves horizontal affine subspaces and $\tilde{F}$ preserves affine vertical subspaces.

We are ready to prove the fundamental geometric property of such $F: X \to Y$: there are uniformly flat sheets and uniformly steep fibers.

\begin{prop} Let $F: X = W_X \oplus V_X \to Y = W_Y \oplus V_Y$ with the hypothesis given in the beginning of the section. The image of each horizontal affine space $v + W_X \subset X$ under $F$ is the graph of a Lipschitz function $\sigma_v: W_Y \to V_Y$. Similarly, the inverse of each vertical affine subspace $g + V_Y \subset Y$ under $F$ is the graph of
a Lipschitz function $\alpha_g : V_X \to W_X$. The Lipschitz constant can be taken to be the same, for all $v \in V_X$, $g \in W_Y$.
\end{prop}

\begin{proof} We prove the result for fibers $F^{-1}(g + V_Y)$:
the statement for sheets $F(v + W_X)$ is easier. Clearly, $\tilde{F}^{-1}( g + V_Y) \subset g + V_Y$, which is taken by the change of variables $\Phi$ to a set of the form $\alpha_g = \{ (F_v)^{-1}(g) + v, v \in V_X \} \subset X$. From the theorem, for every $v \in V_X$, there is a unique $w(v) \in W_X$ for which
$P_Y F(w(v)+v) = g$ --- thus $F^{-1}( g + V_Y) = \{ (F_v)^{-1}(g) + v, v \in V_X \}$. Said differently, $w(v) + v \in \alpha_g$: the set $\{ (w(v), v), v \in V_X \} \subset X$ is the graph of a Lipschitz map.

The uniformity (on $g$) of the Lipschitz constant of the maps $v \mapsto w(v)$  is responsible for the uniform steepness of the fibers.
\end{proof}
\qed

In opposition to the arguments in \cite{BP}, \cite{SC} and \cite{CT} for the smooth case, the geometric statements follow without recourse to implicit function theorems. Notice also that the uniform flatness of sheets and steepness of fibers are a counterpart to (differential) transversality between fibers and horizontal affine spaces in the domain and between sheets and vertical affine spaces in the counterdomain.

The restriction of $F$ to horizontal affine subspaces is injective but the restriction to fibers $\alpha$ is not. In particular, in the standard Ambrosetti-Prodi case, $F$ restricted to each fiber is simply the map $x \in \RR \mapsto -x^2 \in \RR$, after global changes of variables. The theorem becomes evident from this fact, first proved in \cite{BP}.

Vertical lines may be taken by $F$ to the horizontal plane,  indicating yet another relevant transversality property of the fibers. To see this, take $\Omega = [-\pi/2,\pi/2]$ and $F(u) = - u'' - f(u)$, so that $\lambda_2=4$ and the corresponding eigenvector $\varphi_2$ is odd. Set $a=3$ and $b=5$, split $X = W_X \oplus \langle \varphi_2 \rangle$ and take $f(x) = e(x) + 4x$, where $e(x)$ is even (convexity is not necessary!) and $\Ran f' = (a,b)$. By symmetry, we have
\[ \langle \ F(t \varphi_2)\ ,\  \varphi_2 \ \rangle =
\int_\Omega ( - t \,  \Delta \varphi_2 - e(t \varphi_2) - 4 t \, \varphi_2)\  \varphi_2 = 0. \]

\section{Geometry and numerics}

The statements in the previous section are exactly what we need to mimic the numerical algorithms in \cite{CT} for solving $F(u) = g$, i.e., the differential equation (2).
This section emphasizes the points where some alterations are needed, but most details common to the smooth and Lipschitz scenarios, which are provided in \cite{CT}, will not be presented.

\subsection{Finding the right fiber}

To solve $F(u) = g$, first find any point $u = w+v \in \alpha_g$, the fiber associated to $g$. Said differently, find $u$ so that $P_Y F(u) = P_Y g$.
In order to do this, notice that, from the results of the previous section, each fiber $\alpha_g$ intersects each horizontal affine space $v + W_X$ at a single point. Thus each horizontal point $P_Y g \in W_Y$ corresponds to a unique $w(g) \in  W_X$ for which $F_v(w(g)) = P_Y g$.

Said differently, each bi-Lipschitz map $F_v: W_X \to W_Y$ takes a point $w(g)$ in the fiber $\alpha_g$ to a point $P_Y g$, which is the only point in the vertical affine space $g + V_Y$ in the horizontal subspace $W_Y$. In words: there is a bi-Lipschitz map between the set of all fibers (represented by points in $W_X$) to the set of vertical affine spaces (represented by points in $W_Y$).

Now the good (numerical) news: to invert each map $F_v$, simply invert the (bi-Lipschitz) homeomorphism $F_v \circ T^{-1} = I - K$, where $K$ is a contraction, as shown in the proof of Theorem 1! So the approximation of $F_v^{-1}(P_Y G)$ is amenable to standard numerical algorithms. In the differentiable case, we could do somehow better: we could invert by continuation where local steps are given by Newton iterations. In the strictly Lipschitz context, we lose quadratic convergence --- some acceleration techniques are still available, but we provide no details.

\subsection{Moving along a fiber}

Once a point in $\alpha_g$ is identified, we need to learn to walk along the fiber, or more precisely, we have to compute the point in the fiber with a given height $v \in V_X$. This is turn is the main piece of information for a finite dimensional inversion algorithm for the restriction $F: \alpha_g \to g + V_Y$.
This is done by continuation starting from a given point in the fiber $u \in \alpha_g$.

Clearly small perturbations of a point $u_1 \in \alpha_g$ leave the fiber. But the previous algorithm --- more precisely, the inversion of each map $F_v$ --- makes it possible to change $u_1 = w_1 + v_1$ to a point $\tilde{u} = \tilde{w} + v_2$ and then using $\tilde{u}$ as a starting point to solve $F_{v_2}(w_2) = P_Y g$, giving rise to a point $u_2 = w_2 + v_2 \in \alpha_g$.

\subsection{Stability of the decomposition, a numerical necessity}

The uniformity on the flatness of sheets and the steepness of fibers has a relevant consequence for numerics, which has been detailed in \cite{CT}. In a nutshell, whatever algorithm we use requires an approximation for the projection $P_Y: Y \to Y$. This is usually accomplished by computing (approximate) eigenfunctions associated to the eigenvalues of $- \Delta_D$ in the interval $(a,b)$. A concrete possibility is to approximate functions by finite elements. The upshot is that the numerics handles an approximation $\tilde{V}_X=\tilde{V}_Y$ of $V_X=V_Y$ and its orthogonal complement. However, from the uniformities, such spaces, for sufficiently close approximations, still induce global Lyapunov-Schmidt decompositions giving rise to sheets and fibers, for which the algorithms described in the previous paragraphs hold and provide robust approximations to the real answers.

\begin{figure}[ht]
\begin{center}
\epsfig{height=30mm,file=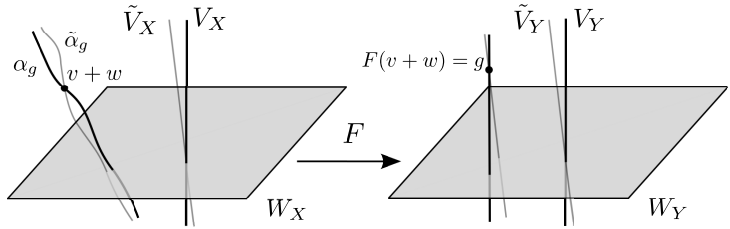}
\end{center}
\caption{\capsize Irrelevant perturbations for numerical purposes}
\label{fig:wil}
\end{figure}

In the figure, we sketch fibers and their perturbations; there is an analogous reasoning for sheets.
In the figure, $\tilde{V}_Y$ changes slightly the affine vertical
subspace through the point $g$, which in turn, when inverted, gives rise to a slightly different fiber $\tilde{\alpha}_g$. Still, $\tilde{\alpha}_g$ is a Lipschitz graph of a function from $\tilde{V}_X$ to $\tilde{W}_X = (\tilde{V}_X)^\perp$. The  (geometric) thing to notice is the fact that it is the steepness of the fibers which ascertain the robustness of the global Lyapunov-Schmidt decomposition.

\end{document}